\documentclass[a4paper,11pt]{amsart}
\usepackage{hyperref}

\usepackage[usenames,dvipsnames,svgnames,table]{xcolor}
\usepackage[all]{xy}
\usepackage{enumitem}
\usepackage[normalem]{ulem}
\newcommand{\tsk}[1]{\textcolor{YellowOrange}}

\usepackage{tikz}
\usepackage{tikz-cd}
\usepackage{graphicx}
\usepackage{pgf,tikz}
\usetikzlibrary{arrows}
\usepackage[left=2cm,right=2cm, top=3cm, bottom=3cm]{geometry}

\makeatletter
\def\@endtheorem{\endtrivlist}
\makeatother

\usepackage[active]{srcltx}
\usepackage{amsmath}
\usepackage{amssymb}
\usepackage{amscd}
\usepackage{amsthm}
\usepackage{mathrsfs}  

\usepackage{nicefrac}

\newcommand{\Pic}{\operatorname{Pic}}
\newcommand{\Cliff}{\operatorname{Cliff}}

\newtheorem{teo}{Theorem}[section]
\newtheorem{defin}[teo]{Definition}
\newtheorem{prop}[teo]{Proposition}
\newtheorem{cor}[teo]{Corollary}

\theoremstyle{definition}

\newtheorem{remark}[teo]{Remark}

\newtheoremstyle{dico}
 {\baselineskip}   
  {\topsep}   
  {}  
  {0pt}       
  {} 
  {.}         
  {5pt plus 1pt minus 1pt} 
  {}          
\theoremstyle{dico}

\numberwithin{equation}{section}

\newcounter{example}[subsection]
\newenvironment{example}[1][]
{\refstepcounter{example}\smallskip\par\noindent\textbf{Example~\theexample.}\nopagebreak\par\noindent}{\medskip}

\newcommand{\ra}{\rightarrow}
\newcommand{\C}{\mathbb{C}}

\newcommand{\Zeta}{{\mathbb{Z}}}
\newcommand{\ZZ}{{\mathbb{Z}}}
\newcommand{\N}{{\mathbb{N}}}

\newcommand{\meno}{^{-1}}

\newcommand{\restr}[1]          {\vert_{#1}}
\newcommand{\Aut}{\operatorname{Aut}}
\newcommand{\Hom}{\operatorname{Hom}}

\newcommand{\om}{\omega}

\renewcommand{\phi}{\varphi}

\newcommand{\sx}{\langle}
\newcommand{\xs}{\rangle}

\newcommand{\PP}{\mathbb{P}^1}

\newcommand{\OO}{\mathcal{O}}

\newcommand{\mg}{\mathcal{M}_g}

\newcommand{\mihi}[1]{}

\newcommand{\Nm}{\operatorname{Nm}}

\newcommand{\e}{\qquad \text{and}\qquad}
\def\fracpart#1{\left\langle#1\right\rangle}
 
\begin{document}
 
\pagestyle{myheadings}

\title{Explicit analysis of Positive dimensional fibres of $ \mathcal{P}_{g,r} $ and Xiao Conjecture}

\author{Gian Paolo Grosselli and Irene Spelta}

\address{{Gian Paolo Grosselli  \\ Universit\`a degli Studi di Pavia  \\ Dipartimento di Matematica \\ Via Ferrata 5}{ 27100 Pavia, Italy.}}
\email{g.grosselli@campus.unimib.it}

\address{Irene Spelta  \\ Universit\`a degli Studi di Pavia  \\ Dipartimento di Matematica \\ Via Ferrata 5  \\ 27100 Pavia, Italy.  }
\email{irene.spelta01@universitadipavia.it}

\date{}

	\begin{abstract}
	We focus on the positive dimensional fibres of the Prym map $\mathcal{P}_{g,r}$. We present a direct procedure to investigate infinitely many examples of positive dimensional fibres. Such procedure uses families of Galois coverings of the line admitting a 2-sheeted Galois intermediate quotient. Then we generalize to families of Galois coverings of the line admitting a Galois intermediate quotient of higher degree and we show that the higher degree analogue of the aforementioned procedure gives all the known counterexamples to a conjecture by Xiao on the relative irregularity of a fibration. 
	\end{abstract}
	\maketitle
	\setcounter{section}{-1}

	\section{Introduction}
Let $  C $ be a curve of genus $g\geq1$ and let $ f : \tilde C \ra C $ be a 2-sheeted covering of $C$ ramified at $r\geq 0$ points. The Prym variety $P:=P(\tilde C, C)$ is a polarized abelian variety of dimension $g-1+\frac{r}{2}$ associated with $ f $. It is defined as the identity component of the kernel of the Norm map $ \Nm_f : J\tilde C \ra JC $. The theta divisor of $J\tilde C$ induces a polarization on the Prym variety $P$ of type $\delta:=(1,\dots,1,2,\dots,2)$ with 2 repeated $ g $ times if $r>0$ and $g-1$ times if $ r=0 $.

Let us denote by $\mathcal{R}_{g,r}$ the coarse moduli space of isomorphism classes of coverings $ f $ and by $\mathcal{A}_{g-1+\frac{r}{2}}^\delta $ the one of abelian varieties of dimension $  g-1+\frac{r}{2} $ with polarization of type $\delta$.
The theory of double coverings provides an alternative description of $\mathcal{R}_{g,r}$. Indeed there is a 1-1 correspondence between double covers $f : \tilde C \ra C $ and triples $(C,\eta, B)$ in \[\mathcal{R}_{g,r}=\{(C,\eta, B): C\in \mg, \eta\in \Pic^{\frac{r}{2}}(C), B \text{ reduced in } \vert \eta^{\otimes 2}\vert\}.\]
When $ r = 0 $, the branch divisor $B$ is empty, hence we can identify $ f $ with a pair $(C,\eta)$, with $\eta\in\Pic^0(C)\smallsetminus \{\OO_C\}$ such that $\eta^{\otimes 2}\cong\OO_C.$

The Prym map is the morphism
\[\mathcal{P}_{g,r}: \mathcal{R}_{g,r}\ra\mathcal{A}_{g-1+\frac{r}{2}}^\delta 	\]
which sends triples $(C,\eta, B)$ to the corresponding Prym variety $ P $.

The case $ r=0 $ is very classical. Indeed unramified Prym varieties are principally polarized abelian varieties and they have been studied for over one hundred years. Our (algebraic) point of view has been presented for the first time by Mumford in \cite{mumford} in 1974. Then many papers investigated this case and nowadays we have a lot of information on $\mathcal{P}_{g,0}$. Donagi very well discusses it in \cite{donagi}. Good surveys are \cite{Farkas} and \cite{ortega}. 

The case $r>0$ has become of interest only more recently. Indeed ramified Prym varieties are abelian varieties no longer principally polarized (except in the case of $r=2$). As such, they started to be studied quite late. The seminal paper is the one by Marcucci and Pirola (\cite{mp}), which came out only in 2012. From this work, many other authors have investigated the ramified Prym maps $\mathcal{P}_{g,r}$. At this moment it is a very active area of research. For instance, very recently, it has been proved that if $r\geq 6$ then  $\mathcal{P}_{g,r}$ is injective (\cite{ikeda}, \cite{naranjo-ortega2}). 

When the genus $ g $ and the number $ r $ are low, more precisely when $g<6$ for $r=0$, $g\leq 4$ for $r=2$ and $g\leq 2$ for $r=4$, the fibres of the Prym maps $\mathcal{P}_{g,r}$ are positive dimensional and they carry plenty of geometry which is well-understood (see \cite{donagi} and \cite{fns}). The structure of the generic positive dimensional fibre is so peculiar that one needs to find an ad-hoc procedure to describe each of them.

In \cite{naranjo}, Naranjo stated that, for $g$ large enough, the \'etale Prym map $\mathcal{P}_{g,0}$ has positive dimensional fibres only on the locus of coverings of hyperelliptic curves and on some components of the locus of coverings of bielliptic curves. Naranjo and Ortega
(see \cite[Theorem 1.2]{naranjo-ortega2}) showed that the ramified Prym map $\mathcal{P}_{g,r}$, with $r=2,4$ and any value for $g$, has positive dimensional fibres when restricted to covers of hyperelliptic curves. In the same paper, the authors also proved that the Prym map $\mathcal{P}_{5,2}$ carries positive dimensional fibres when restricted to the locus of trigonal curves. Finally, Casalaina-Martin and Zhang \cite[Theorem 5.1]{cas zhang} produced some positive dimensional fibres for $\mathcal{P}_{3,4}$ under the assumption $\eta$ effective.

Thus it turns out that, except for few isolated cases, the hyperelliptic locus represents a good place to look for positive dimensional fibres for the Prym map. Nevertheless, when $r>0$ it is unknown if we should expect the existence of other examples. Indeed we only have the following:
\begin{prop}[\cite{naranjo-ortega2}]
	Assume $r>0$. If the differential $d\mathcal{P}_{g,r}$ is not injective then we are in one among these cases.
	\begin{itemize}
		\item[$r=2:$] $C$ is hyperelliptic or trigonal or plane quintic or $\eta=\OO_C(x+y-z)$, for $x,y,z \in C$.
		\item[$r=4:$] $C$ is hyperelliptic or $h^0(C,\eta)>0$. 
	\end{itemize}
\end{prop}
This Proposition doesn't conclude the classification. Apart from the hyperelliptic locus and the isolated examples mentioned in the previous paragraph (i.e.\ the trigonal locus in genus 5 and $	\eta$ effective in genus 3), we still do not known the behaviour of the differential $d\mathcal{P}_{g,r}$ on the remaining cases.

The goal of this paper is to analyse infinitely many examples of positive dimensional fibres of the Prym maps, both in the \'etale or in the ramified case. We use positive dimensional families of Galois covers $\tilde C\ra \tilde C/\tilde G\cong \PP$ of the line where the genus $\tilde g:=g(\tilde C)$, the number of ramification points $s$ and the monodromy are fixed. Then we look for which among these families admit as intermediate quotient $\tilde C\ra C$ a 2:1 map ramified in $r=0,2,4$. Hence we select the ones with associated $d\mathcal{P}_{g,r}$ non-injective. 
This request corresponds to a simple numerical condition as follows.
\begin{prop}\label{ prop differenziale}
	If $\dim(S^2H^0(\om_{\tilde C}))^{\tilde G}- \dim(S^2H^0(\om_{C}))^{\tilde G}<\dim H^0(\om_{\tilde C}^{\otimes2})^{\tilde G}$ then the differential  $d\mathcal{P}_{g,r}$ along the family is not injective. Hence the Prym map has positive dimensional fibres along the family.
\end{prop}
In particular, when $\dim(S^2H^0(\om_{\tilde C}))^{\tilde G}= \dim(S^2H^0(\om_{C}))^{\tilde G}$ the family is all contained in a fibre of $\mathcal{P}_{g,r}$. This is always the situation in case of $H^0(\om_{\tilde C}^{\otimes2})^{\tilde G}=1$, i.e.\ of 1-dimensional families, under the assumption of Proposition \ref{ prop differenziale}. 

We show the following:
\begin{teo}
	For any $N\in \mathbb{N}$ there exist 1-dimensional families of Galois covers $\tilde C\ra C\ra \PP$ contained in the fibres of $\mathcal{P}_{2N,0}, \mathcal{P}_{2N,2}, \mathcal{P}_{2N-1,0}, \mathcal{P}_{2N-1,4}$. 
\end{teo}
Easily we show that (unfortunately) all such families arise as coverings of curves $C$ lying in the hyperelliptic locus.\\

In general, starting from a family of curves, one can construct fibrations that have the curves of the family as fibres. In particular, we focus on those obtained from families under the assumption of Proposition \ref{ prop differenziale}. At the same time, we generalize to families of Galois coverings $\tilde C\ra \tilde C/\tilde G\cong \PP$ (with fixed genus $\tilde g:=g(\tilde C)$ and monodromy) admitting  Galois intermediate quotient $\tilde C\ra C$ of degree $d\geq 2$. Accordingly, we give the definitions of Prym variety $P(\tilde C, C)$ and of higher degree Prym map $\mathcal{P}_{g,r}(d)$.

First, we show that a higher-degree analogue of Proposition \ref{ prop differenziale} holds. Indeed we prove the following:
\begin{prop}\label{porp differenziale higher degree}
	If $\dim(S^2H^0(\om_{\tilde C}))^{\tilde G}- \dim(S^2H^0(\om_{C}))^{\tilde G}<\dim H^0(\om_{\tilde C}^{\otimes2})^{\tilde G} $ then the differential of the Prym map  $\mathcal{P}_{g,r}(d)$ along the family is not injective. Hence the Prym map has positive dimensional fibres along the family.
\end{prop}

It appears that fibrations $h: S\ra B$ obtained from families under the assumption of Proposition \ref{porp differenziale higher degree} are quite interesting. Indeed, using them, we produce all the known counterexamples to a conjecture of Xiao. 

In \cite{xiao}, Xiao proved that a non trivial fibration with base curve $B\cong \PP$, general fibre of genus $\tilde g$ and irregularity $q$, satisfies $q\le\frac{{\tilde g}+1}2.$
Furthermore, for $g(B)>0$, he conjectured that the relative irregularity of the fibration $q_h$ satisfies $q_h\le\frac{{\tilde g}+1}2.$ The four known counterexamples have been constructed in \cite{pirola} and in \cite{albano pirola} as fibrations associated with families of cyclic prime odd \'etale covers of hyperelliptic curves (elliptic curves in case of \cite{pirola}) carrying a non-injective differential $d\mathcal{P}_{g,r}(d)$. in particular, the three examples of \cite{albano pirola} fit perfectly the structure of our families: a Theorem by Ries (\cite{ries}) guarantees that they yield dihedral Galois cover of $\PP$, i.e.\ they provide towers $\tilde C\xrightarrow{d:1} C\ra \tilde C/D_d\cong \PP$. 

In light of Proposition \ref{porp differenziale higher degree}, it seemed natural to us to check if there exist other families of Galois coverings $\tilde C\xrightarrow{d:1} C\ra \PP$ with non-injective differential and which disprove the conjecture. Notice that we do not require $C$ to be hyperelliptic and we admit any Galois intermediate quotient  $\tilde C\xrightarrow{d:1}  C$ of any degree.  By means of computer calculations (our \verb|MAGMA| script is available at \textit{http://mate.unipv.it/grosselli/publ/}), we show the following
\begin{prop}\label{controesempi Xiao}
	Up to $\tilde g=12, s=14$ (i.e. dimension 11), the only positive dimensional families of Galois towers $\tilde C\xrightarrow{d:1} C\ra \PP$ carrying non-injective differential and disproving Xiao's conjecture are the one of \cite{pirola} and the ones of \cite{albano pirola}.
\end{prop}

The third example of \cite[Theorem 1.2]{albano pirola} is obtained via a degeneration argument. Therefore one of the four examples that we find with Proposition \ref{controesempi Xiao} is presented in a slightly different way from the original, although it is clearly the same. For this reason, we think it may be useful to give a very brief description of all the examples.

The paper is organized as follows: in Section \ref{sez 1} we give an overview on Prym maps while in Section \ref{sez 2} we explain our constructive method. Finally, in Section \ref{sez 3}, we recall something on Xiao fibrations and we describe the counterexamples. \\

{\bfseries \noindent{Acknowledgements}}. We are indebted to P. Frediani for suggesting to us the topic of this collaboration and for the many fruitful discussions we had. 
	\section{The state of the art}\label{sez 1}

In this section we would like to overview the literature on positive dimensional fibres of the Prym maps $\mathcal{P} _{g,r}: \mathcal{R}_{g,r}\ra\mathcal{A}_{g-1+\frac{r}{2}}^\delta $. We recall that \begin{equation}\label{dimensioni}
\dim \mathcal R_{g, r}=3g-3+r\quad \text{and}\quad  \dim \mathcal A_{g-1+\frac r2}^{\delta}=\frac{1}{2}\left(g-1+\frac{r}{2}\right)\left(g+\frac{r}{2}\right).
\end{equation} 

First, let us briefly recall the classical case, that is $r=0$ (the standard notation refers to $\mathcal R_{g, 0}$ as to $ \mathcal R_{g} $, the same for $\mathcal P_{g,0}$). Easily from \eqref{dimensioni} we see that the generic fibre of $\mathcal{P}_g$ is positive dimensional when $g\leq 5$. A detailed study of its geometric structure is provided by the works of Verra for $g=3$ (\cite{verra}), Recillas for $g=4$ (\cite{recillas}) and Donagi for $g=5$ (\cite{donagi}). Cases with $g=1, g=2$ are summarized in \cite[Section 6]{donagi}. When $g=6$ the fibre is generically finite of degree 27 (\cite{ds}). On the other hand, when $g\geq 7$, the Prym map is generically injective but never injective (see \cite{donagi} and references therein). 

The positive dimensional fibres of $\mathcal{P}_{g}$ are characterized as follows:
\begin{teo}[Mumford \cite{mumford}, Naranjo \cite{naranjo}]
	Assume $g\geq 13$. Then $\mathcal{P}_{g}$ has positive dimensional fibres at $(\tilde C, C)$ if and only if $C$ is either hyperelliptic or it belongs to one among the components of the bielliptic locus where $\tilde C$ carries $\mathbb{Z}/2\times \mathbb{Z}/2\subseteq \Aut(\tilde C)$. 
\end{teo}

Now we focus on the ramified cases, i.e.\ $r>0$. The inequality $\dim \mathcal R_{g, r}> \dim \mathcal A_{g-1+\frac r2}^{\delta}$ is satisfied only in six cases, that is $r=2$ with $1\leq g\leq 4$ and $r=4$ with $1\leq g\leq 2$. All of them are considered in \cite{fns}. Indeed it is shown the following:
\begin{prop}(\cite[Proposition 1.2 and Corollary 1.3]{fns})\label{Fibre dim pos gen bassi}
	Under the assumptions \[
	(g,r) \in \{(1,2),(1,4),(2,2),(2,4),(3,2),(4,2)\},
	\]
	the ramified Prym map  $\mathcal P_{g,r}$ is dominant. Therefore the generic fibre $F_{g,r}$ of $\mathcal P_{g,r}$ has $
	\dim F_{1,2}=1, \, \dim F_{2,2}=2, \,\dim F_{3,2}=2,\dim F_{4,2}=1, \,\dim F_{1,4}=1, \,\dim F_{2,4}=1.$ 
\end{prop}
Hence the paper gives a detailed description of the generic fibre for all the six cases (see \cite[Theorem 0.1]{fns}). 

Let us now focus on $\dim \mathcal R_{g, r}\leq\dim \mathcal A_{g-1+\frac r2}^{\delta}.$ The first result is the following: 
\begin{teo}[Marcucci-Pirola \cite{mp}, Marcucci-Naranjo \cite{mn}, Naranjo-Ortega \cite{naranjo-ortega}]\label{teo gen iniett}
	The ramified Prym map is generically injective as far as the dimension of $\mathcal R_{g, r}$ is less than or equal to the dimension of $\mathcal A_{g-1+\frac r2}^{\delta}.$
\end{teo}
Actually, the equality between the dimensions is reached only in the case of $g=3$ and $r=4$ where more it is known:
\begin{teo}[Nagaraj-Ramanan \cite{nr}]
	Let $g\geq3$. The Prym map $\mathcal P_{g, 4}$ restricted to the locus of tetragonal curves has generically degree 3.
\end{teo}

Quite recently Theorem \ref{teo gen iniett} has been improved:

\begin{teo}[Ikeda \cite{ikeda} for $g=1$, Naranjo-Ortega \cite{naranjo-ortega2} for all $g$]
	The Prym map $\mathcal P_{g, r}$ is injective with injective differential for all $r\geq6$ and $g>0$.
\end{teo}
Thus we will look for the positive dimensional fibres of $\mathcal P_{g, 2}$ and $\mathcal P_{g, 4}$. Since $\mathcal P_{1, 2}$ and $\mathcal P_{1, 4}$ have positive dimensional generic fibre, we will assume $g\geq 2$. 

The codifferential of $\mathcal P_{g,r}$ at a point  $[(C, \eta, B)]\in \mathcal R_{g,r}$  is given by the multiplication map (\cite{mp})
\begin{equation}\label{codifferenziale}
	d\mathcal P_{g,r}^* (C, \eta, B): S^2 H^0(C, \omega_C \otimes \eta) \ra H^0(C, \omega_C^2 \otimes \mathcal O(B)).
\end{equation}

We have the following:
\begin{prop}[\cite{naranjo-ortega2}]
	Let $L:=\omega_C\otimes\eta$. If $
	d\mathcal P_{g,r}$ is not injective at $[(C, \eta, B)]$ then one of the following holds:
	\begin{enumerate}
		\item  $L$ is not very ample or
		\item $L$ very ample and 
		\begin{enumerate}
			\item $r=2$ and $\Cliff(C)\leq 1$ or 
			\item $r=4$ and $\Cliff(C)=0$.
		\end{enumerate}
	\end{enumerate}
	\begin{proof}
		The proof is a straightforward application of Green-Lazarsfeld criterion for the surjectivity of a multiplication map (see \cite[Theorem 1]{greenlaz}): the map $d\mathcal P_{g,r}^*: S^2 H^0(C,L)\ra H^0(C, L^2)$ is surjective if $L$ is very ample and $\deg L \geq 2g+1-2h^1(C,L)-\Cliff(C)$. Since \[\deg L=2g-2+\frac{r}{2}\e h^1(C,L)=h^0(C, \omega_C\otimes L\meno)=0 \]
		we conclude.
	\end{proof}
\end{prop}
As already observed in \cite[Remark 2.2]{naranjo-ortega2}, the above Proposition can be rephrased as follows:
\begin{prop}\label{elenco 2}
	If the differential $d\mathcal P_{g,r}$ is not injective at $[(C, \eta, B)]$ then:
	\begin{enumerate}
		\item  $r=2$ and $\eta=\OO(x+y-z)$ for $x,y,z\in C$ or $r=4$ and $h^0(C,\eta)>0$. Otherwise 
		\item $r=2$ and $C$ is hyperelliptic, trigonal or a quintic plane curve or $r=4$ and $C$ is hyperelliptic.
	\end{enumerate}
	\begin{proof}
		(1) is borrowed from \cite[Lemma 2.1]{lange-ortega Cyclic cov} while (2) follows from the definition of the Clifford Index.
	\end{proof}
\end{prop}
Now we list evidence of positive dimensional fibres that we find in the literature. For a proof of these results, we refer to the cited papers.
\begin{prop}(Naranjo-Ortega, \cite[Theorem 1.2]{naranjo-ortega2})
	Let $\mathcal P_{g,r}^h$ be the restriction of $\mathcal P_{g,r}$ to the locus of coverings of hyperelliptic curves of genus $g$ ramified in $r$ points ($ r=2,4 $). Then the generic fibre of $\mathcal P_{g,2}^h$, respectively of $\mathcal P_{g,4}^h$, is birational to a projective plane, respectively to an elliptic curve.
\end{prop}
\begin{remark}
	When $g=2$ the restriction $\mathcal P_{2,r}^h$ coincides with $\mathcal P_{2,r}$. Indeed the paper \cite{fns} studies $\mathcal P_{2,2}$, respectively $\mathcal P_{2,4}$, and it shows that the generic fibre is isomorphic to a plane minus 15 lines, respectively to an elliptic curve minus 15 points.
\end{remark}
\begin{prop}(Naranjo-Ortega, \cite[Proposition 2.4]{naranjo-ortega2})
	Let $\mathcal P_{5,2}^{tr}$ be the restriction of $\mathcal P_{5,2}$ to the locus of coverings of trigonal curves of genus $5$ ramified in $2$ points. Then the fibres are all positive dimensional.
\end{prop}
\begin{prop}(Casalaina Martin-Zhang, \cite[Theorem 5.1]{cas zhang})
	Let $\mathcal R_{3,4}^{Eck}\subset\mathcal R_{3,4}$ be the subset of triplets $(C,\eta, B)$ such that $C$ is a smooth quartic plane curve, $\eta$ is given by the class of a bitangent and $B$ is a reduced divisor on $C$ cut by a line $l$ in the plane where $C$ is canonically embedded. Then the generic fibre of $\mathcal P_{3,4}$ restricted to $\mathcal R_{3,4}^{Eck}$ is isomorphic to the elliptic curve described as the covering of $l$ ramified on $B$. 
\end{prop}
\begin{remark}
	Notice that here the positive dimensional fibres are realized under the assumption $\eta$ effective. 
\end{remark}

\section{The Procedure}\label{sez 2}
In this section, we describe our strategy to investigate infinitely many positive dimensional fibres of the Prym maps. 
In particular, we study families of towers of Galois covers $\tilde C\xrightarrow{2:1}C\to\PP$. 
Our procedure doesn't bound the genus of the curves occurring in such towers. For this reason, we are able to describe infinitely many examples.

In order to do this, we introduce the Prym datum as given in \cite{cfgp,fredi gross}.
We recall the definition.
\begin{defin}
	A  Prym datum of type $(s,r)$ is a triple  $\Xi=(\tilde G,\tilde \theta,\sigma)$, 
	where $\tilde G$ is a finite group, $\tilde \theta:\Gamma_s\to\tilde G$ is an  epimorphism,  $\sigma\in\tilde G$ is a central involution and 
	$$r=\sum_{i:\sigma\in\sx{\tilde \theta(\gamma_i)}\xs}\frac{|\tilde G|}{\operatorname{ord}(\tilde \theta(\gamma_i))}.$$
\end{defin}
This datum corresponds to a family of Galois coverings and the generic point of the family fits in the following diagram
\begin{equation}
\label{tc-c}
\begin{tikzcd}[row sep=tiny]
\tilde{C} \arrow{rr}{f} \arrow{rd}{\pi} & &  \arrow{ld } C   = \tilde{C} /\sx \sigma\xs  \\
& \PP &
\end{tikzcd}
\end{equation}
where $\tilde C$ and $C$ are curves of genus $\tilde g$ and $g$ respectively. $G:=\tilde G/\sx\sigma\xs$ is the quotient group acting on $C$, the composition of $\tilde \theta$ with the projection to the quotient is an epimorphism $\theta:\Gamma_s\to G$. 
The maps $\tilde \theta$ and $\theta$ are respectively the monodromies of the two Galois covers $\tilde C\to\PP=\tilde C/\tilde G$ and $C\to\PP=C/G$. The cover $\tilde C\to\PP$ is branched on $s$ points. Moreover, in all the examples we consider, $s=4$ and $C\ra \PP$ is branched on the same points of $\tilde C\to\PP$. 
The cover $f$ is 2-sheeted and branched in $r$ points.

There is a natural identification between the tangent space to the family at the generic point and the space of the infinitesimal deformations of $\tilde C$ that preserve the action of $\tilde G$. The latter is isomorphic to $H^1(\tilde C, T_{\tilde C})^{\tilde G}$ $((\cong H^0(\tilde C,\omega_{\tilde C}^2)^{\tilde G})^*)$.
Thus $\dim H^0(\tilde C,\omega_{\tilde C}^2)^{\tilde G}=s-3 $ equals the dimension of the family.

Recall that $\sigma$ gives a decomposition of $V:=H^0(\tilde C,\omega_{\tilde C})$ in $\pm 1$-eigenspaces, resp. $V_+$ and $V_-$, where $V_+\cong H^0(C,\omega_C)$ and $V_-\cong H^0(C,\omega_C\otimes\eta)$.
Similarly we can define $W:=H^0(\tilde C,\omega_{\tilde C}^2)$ and get a decomposition $W=W_+\oplus W_-$ with $W_+\cong H^0(C,\omega_C^2\otimes\eta^2)=H^0(C,\omega_C^2\otimes\OO(B))$.
Let us denote 
\begin{equation}\label{NtildeN}
\tilde N:=\dim(S^2H^0(\tilde C,\omega_{\tilde C}))^{\tilde G}\qquad\text{and}\qquad N:=\dim(S^2H^0(C,\omega_C))^{G}.
\end{equation}
Immediately, we have $\tilde N -N=\dim(S^2V_-)^{\tilde G}$.

We are interested in families that lie in positive dimensional fibres of the Prym map, so we look for a condition that makes the codifferential of the Prym map not surjective. 
We have the following
\begin{prop}\label{Ntilde-N}
	%
	If $\tilde N - N<s-3$ the differential of the Prym map $d\mathcal{P}_{g,r}$ along the family is not injective, hence the Prym map has positive dimensional fibres along the family. 
\end{prop}
\begin{proof}
	As in \eqref{codifferenziale}, the codifferential of the Prym map at the generic element of the family is the multiplication map 
	\[
	m = (d\mathcal{P}_{g,r})^* : (S^2 V_-)^{\tilde G} \to W_+^{\tilde G}.
	\]
	Since by assumption  $\dim (S^2 V_-)^{\tilde G}=\tilde N-N<s-3=\dim W_+^{\tilde G}$ the codifferential cannot be surjective hence its dual has a non-trivial kernel.
\end{proof}
\begin{cor}
	If $\tilde N=N$ the family is contained in a fibre of the Prym map.
\end{cor}

Now we present our construction. It produces infinitely many Prym data. They yield 1-dimensional families of Galois towers of type \eqref{tc-c} that carry constant Prym variety. Thus they lie in positive dimensional fibres of the Prym map.
We organize these data into 5 classes as follows.
\begin{teo}\label{teoclassi}
	For any $N\in\N$ there are 1-dimensional families of Galois covers $\tilde C\to C\to\PP$ in the fibres of the Prym maps $\mathcal{P}_{2N,2}$, $\mathcal{P}_{2N-1,0}$ (2 families), $\mathcal{P}_{2N,0}$ and $\mathcal{P}_{2N-1,4}$.
	All the families lie in their respective hyperelliptic locus.
\end{teo}

All the families carry an abelian Galois group, so we can refer to \cite[Section 4]{fredi gross mohj}.
We describe in detail the first case.

Fix a positive integer $N$, set the odd number $k=2N+1$ and denote $C_n\cong \Zeta/n\Zeta$.
The Prym datum is defined by the group $\tilde G=C_2\times C_{2k}\subset C_{2k}^2$ under the inclusion $\begin{pmatrix}
1\\0
\end{pmatrix}\mapsto \begin{pmatrix}
k\\0
\end{pmatrix}$, $\begin{pmatrix}
0\\1
\end{pmatrix}\mapsto \begin{pmatrix}
0\\1
\end{pmatrix}$,  the monodromy $\tilde\theta:\Gamma_4\to\tilde G$ is represented by the matrix
\[A=\begin{pmatrix}k&0&0&k\\0&k&2&k-2\end{pmatrix}\]
where the $i$-th column is $\tilde\theta(\gamma_i)$ under the inclusion in $C_{2k}^2$ 
and the involution $\sigma=(k,k)^t$.

In case of abelian groups, the character group $\tilde G^*=\Hom(\tilde G,\C)$ is isomorphic to $\tilde G$.
In our situation, a character in $\tilde G^*$ can be identified with an element $n=(n_1,n_2)$ where $n_1\in C_2$ and $n_2\in C_{2k}$.
Set $V=H^0(\tilde C,\omega_{\tilde C})$ and let $V=V_+\oplus V_-$ be the eigenspace decomposition induced by the action of $\sigma$.
As before,  $V_+\cong H^0(C,\omega_C)$.

Our target is to compute the dimension of $(S^2H^0(\tilde C,\omega_{\tilde C})_-)^{\tilde G}$ and to show that it is zero. 
In this way, the codifferential of the Prym map would be trivial on the family and thus the Prym map would be constant.

Let $H^0(\tilde C,\omega_{\tilde C})_n$ be the subspace of $H^0(\tilde C,\omega_{\tilde C})$ where $\tilde G$ acts via the character $n$, and denote by $d_n$ its dimension.
By \cite[Prop.\ 2.8]{mohj zuo}, the dimension for a non trivial character $n$ is given by
\[d_n=-1+\sum_{i=1}^4\fracpart{-\frac{\alpha_i}{2k}}\]
where $\fracpart x$ denotes the fractional part of a real number $x$ and 
\[\alpha=(\alpha_1,\alpha_2,\alpha_3,\alpha_4):=(n_1,n_2)\cdotp A=\begin{pmatrix}kn_1,&kn_2,&2n_2,&kn_1+(k-2)n_2\end{pmatrix}.\]

We remind that if $x\in\mathbb{R}\smallsetminus\mathbb{Z}$ then $\fracpart{x}+\fracpart{-x}=1$. It is straightforward that for any $n$ such that $2n=0$ we have $d_n=0$. Indeed 
\[d_n=-1+\fracpart{\frac{-n_1}{2}}+ \fracpart{\frac{-n_2}{2}}+\fracpart{\frac{-n_2}{k}}+\fracpart{-\frac{kn_1+(k-2)n_2}{2k}},\]
therefore when $n_1=0, n_2=k$ 
\[d_n= -1+0+\fracpart{-\frac{n_2}{2}}+0+\fracpart{-\frac{n_2}{2}}=0,\]
and when $n_1=1,\;  n_2=0,k$ 
\[ d_n=-1+\frac{1}{2}+ \fracpart{\frac{-n_2}{2}}+0+\fracpart{-\frac{1+n_2}{2}}=0. \]

So now we suppose $-n\ne n$, i.e.\ $-n_2\ne n_2$, so $n_2\ne 0,k$.

\begin{itemize}
	\item If $n_1=0$ and $n_2$ is even, then $d_n=-1+\fracpart{-\frac{n_2}{k}}+\fracpart{-\frac{(k-2)n_2}{2k}}=0$.
	\item If $n_1=0$ and $n_2$ is odd, then $d_n=-1+\frac{1}{2}+\fracpart{-\frac{n_2}{k}}+\fracpart{-\frac{(k-2)n_2}{2k}}$. Thus $d_n+d_{-n}=1$, hence exactly one between $d_n$ and $d_{-n}$ is 1 and the other is 0. In both cases the product $d_{n}d_{-n}$ is zero.
	Moreover the sum of all $d_n$ of this kind is $\frac{k-1}{2}$.
	\item If $n_1=1$ and $n_2$ is even, then $d_n=-1+\frac{1}{2}+\fracpart{-\frac{n_2}{k}}+\fracpart{-\frac{k+(k-2)n_2}{2k}}$. As in the previous case $d_n+d_{-n}=1$ and $d_nd_{-n}$ is always 0. Therefore the sum of all $d_n$ of this kind equals $\frac{k-1}{2}$.  
	\item If $n_1=1$ and $n_2$ is odd, then $d_n=-1+\frac{1}{2}+\frac{1}{2}+\fracpart{-\frac{n_2}{k}}+\fracpart{-\frac{k+(k-2)n_2}{2k}}=1$. 
	So all $k-1$ terms of this form are 1, hence there are $\frac{k-1}{2}$ couples such that $d_nd_{-n}=1$.
\end{itemize}
The decomposition $H^0(\tilde C,\omega_{\tilde C})=\bigoplus_n H^0(\tilde C,\omega_{\tilde C})_n$ gives us
\[\tilde g=g(\tilde C)=\dim H^0(\tilde C,\omega_{\tilde C})=\sum_n d_n=2k-2.\]
Moreover we have
\[\dim (S^2H^0(\tilde C,\omega_{\tilde C}))^{\tilde G}=\frac{1}{2}\sum_{2n\ne 0}d_nd_{-n}=\frac{k-1}{2}=N.\]

As explained in \cite[Lemma~4.3]{fredi gross mohj} the terms of $H^0(\tilde C,\omega_{\tilde C})$ which are invariant for the action of $\sigma$ are those such that $n_1+n_2$ is even.
Hence the genus $g$ of $C$ is
\[ g=\dim H^0(\tilde C,\omega_{\tilde C})_+=\sum_{n_1+n_2\text{ even}} d_n=k-1\]
and by Riemann-Hurwitz formula we obtain that the ramification of the cover is
\[r=2\tilde g-2-2(2g-2)=2(2k-2)-4(k-1)+2=2.\]
Finally we compute the dimension of $(S^2V_+)^{\tilde G}$:
\[\dim (S^2H^0(\tilde C,\omega_{\tilde C})_+)^{\tilde G}=\frac{1}{2}\sum_{n_1+n_2\text{ even}}d_nd_{-n}=\frac{k-1}{2}=N.\]
This gives $(S^2H^0(\tilde C,\omega_{\tilde C})_-)^{\tilde G}=0$ and thus it allows to conclude.

It only remains to observe that the unique element of $G=\tilde G/\sx\sigma\xs\cong C_{2k}$ of order 2 gives the hyperelliptic involution of $C$.

In the following table, we summarize all the examples outlined in Theorem~\ref{teoclassi}.
For any integer $N$ we define $k$ and consequently, we give the data.
The first line corresponds to the example described above.
As seen $(S^2H^0(\tilde C,\omega_{\tilde C})_+)^{\tilde G}=N$. The same holds for the remaining families.
Since computations are almost identical, except case (3) which is slightly more tricky, we do not repeat them.

\begin{center}
	\begin{tabular}{cccccccc}
		\emph{n}&$k$&$\tilde g$&$g$&$r$&$\tilde G$ 
		&$A$ &$\sigma$
		\\\hline
		(1) & $2N+1$ &$2k-2$&$k-1$&$2$&$C_2\times C_{2k}$
		&$\begin{pmatrix}k&0&0&k\\0&k&2&k-2\end{pmatrix}$ & $\begin{pmatrix}k\\k\end{pmatrix}$
		\\
		(2) & $2N-1$ &$2k-1$&$k$&$0$&$C_2\times C_{2k}$
		&$\begin{pmatrix}k&0&0&k\\0&k&1&k-1\end{pmatrix}$ & $\begin{pmatrix}k\\k\end{pmatrix}$
		\\
		(3) & $N$ &$4k-3$&$2k-1$&$0$&$C_2\times C_{2k}$
		&$\begin{pmatrix}k&k&0&0\\k-1&k-1&1&1\end{pmatrix}$ & $\begin{pmatrix}k\\k\end{pmatrix}$
		\\
		(4) & $2N$ &$2k$&$k$&$0$&$C_2\times C_{2k}$
		&$\begin{pmatrix}k&0&0&k\\0&k&1&k-1\end{pmatrix}$ & $\begin{pmatrix}k\\k\end{pmatrix}$
		\\
		(5) & $2N$ &$2k-1$&$k-1$&$4$&$C_{2k}$
		&$\begin{pmatrix}1&1&k-1&k-1\end{pmatrix}$ & $\begin{pmatrix}k\end{pmatrix}$
		
	\end{tabular}
\end{center}

In order to find other (possibly higher dimensional) families contained in the fibres of $\mathcal{P}_{g,r}$ we use a \texttt{MAGMA} script similar to the one described in \cite{fredi gross} and \cite{fredi gross mohj}.
In these papers, the authors look for Shimura subvarieties generically contained in the Prym locus.
For this reason, they require the differential of the Prym map to be an isomorphism.
Here we want exactly the converse, so we look for data that satisfy Proposition~\ref{Ntilde-N}.
In this way, \texttt{MAGMA} produces various examples of families of different dimensions and genus.
While in the higher dimensional case we only find some sporadic examples, in the 1-dimensional case 
we realized that the examples behaved cyclically in the same way. This motivated our choice to organize them into five classes as in the above table.

\section{The Xiao conjecture}\label{sez 3}
In this section we look at the higher-degree analogue of the Prym maps $\mathcal{P}_{g,r}$. Indeed nothing prevents us to consider Galois covers of curves $C\in \mg$ of degree greater than 2. The theory of Prym varieties easily extends to these cases (see \cite{lange-ortega Cyclic cov} for cyclic covers): to any Galois covering $f: \tilde C\xrightarrow{d:1} C$ one can associate an abelian variety $P:=P(\tilde C,C)$ defined as the connected component to the origin of the kernel of the Norm map $\Nm: J\tilde C\ra JC$. Letting $\tilde g$, resp. $g$, be the genus of the curves $\tilde C$ and $C$, then $P$ has dimension $\tilde g-g$ and inherits a (non-principal) polarization $L$ from the theta divisor associated with $J\tilde C$. 

We denote by $\mathcal{R}(K, g, r)$ the Hurwitz scheme parametrizing coverings $f$: $K$ is the Galois group acting on $\tilde C$, $g$ is the genus of the quotient curve and $r$ the number of branch points. Thus we can define the Prym map: 
\begin{equation}\label{prym degree alto}
\begin{tikzcd}[row sep=0,column sep=small]
	\mathcal{P}_{g,r}(d)\ar[r,phantom,":"]&[-5pt] \mathcal{R}(K, g, r)\rar& \mathcal{A}_{\tilde g -g}^\delta\\
	&{[f]}\ar[r,mapsto]&{[P, L]}
\end{tikzcd}
\end{equation} Let us fix a group $\tilde G\subseteq\Aut(\tilde C)$ with $K\lhd \tilde G$ normal subgroup. For computational reasons, we assume that $\tilde C$ has Galois quotient $\tilde C/\tilde G$ isomorphic to $\PP$ and we let $s$ be the number of branch points. This means that we focus on towers $\tilde C\ra C\ra \PP$. By moving the branch points in $\PP$, we get a family of dimension $s-3$. This family naturally gives rise to a family of the same dimension contained in $\mathcal{R}(K, g, r)$. For more details, we refer the reader to Section \ref{sez 2}, where such construction is considered in case of covers $f$ of degree 2.

As already seen, the tangent space to the family at the generic point is $H^1(\tilde C, T_{\tilde C})^{\tilde G}$. For the same reason, the space of the infinitesimal deformations of $(P, L)$ that preserve the action of $\tilde G$ is $S^2H^{0,1}(P)^{\tilde G}$. If we set $V:=H^0(\tilde C, \om_{\tilde C})$ and we decompose $V=V_+\oplus V_-$, where we identify $V_+:=H^0(\tilde C, \om_{\tilde C})^K(\cong H^0(C, \om_{C}) )$ and $V_-:=H^{1,0}(P)$, we get that the differential of the Prym map $d\mathcal{P}_{g,r}(d)$ yields the map\[d\mathcal{P}_{g,r}(d): H^1(\tilde C, T_{\tilde C})^{\tilde G}\ra S^2H^{0,1}(P)^{\tilde G}. \]
Dualizing we get
\[d\mathcal{P}_{g,r}(d)^*=m\restr{S^2H^{1,0}(P)^{\tilde G}}: S^2H^{1,0}(P)^{\tilde G}\ra H^0(\tilde C, \om_{\tilde C}^{\otimes2})^{\tilde G},  \]
where, as usual, the multiplication map $m: S^2H^0(\tilde C, \om_{\tilde C}) \ra H^0(\tilde C, \om_{\tilde C}^{\otimes2})   $ is the codifferential of the Torelli map.

As in \eqref{NtildeN}, we define $\tilde N, N$. We have a higher degree analogue of Proposition \ref{Ntilde-N}: 
 \begin{prop}\label{proposizione diff non inj}
 	If $\tilde N - N<s-3$, the differential of the Prym map $d\mathcal{P}_{g,r}(d)$ along the family is not injective, hence the Prym map has positive dimensional fibres along the family.  
 	\end{prop}
 \begin{proof}
 	The proof works exactly like the one of Proposition \ref{Ntilde-N}. 
 \end{proof}
Starting from our families of curves, we can construct fibrations $h: S\ra B$ that carry $\tilde C$ as fibres. Indeed the closure of the image of the modular map $t\mapsto[\tilde C_t]$ gives a curve in $\mathcal{M}_{\tilde g}$. Then, up to resolving singularities and taking pull-backs, we get a fibration $h: S\ra B$, as claimed. 

The \emph{irregularity} of the surface $S$ is $q:=\dim H^1(S,\OO_S)$, and $q_h=q-g(B)$ is called \emph{relative irregularity} of the fibration.
It is quite famous that Xiao in  \cite{xiao} proved the following
\begin{teo}[Xiao]
If $h$ is not isotrivial and $B=\PP$ then
\[q\le\frac{{\tilde g}+1}2.\]
\end{teo}
Furthermore he conjectured that, for a base $B$ of positive genus, the relative irregularity of the fibration should satisfy
\begin{equation}\label{congettura Xiao disuguaglianza}
	q_h\le\frac{{\tilde g}+1}2.
\end{equation}
It is known that the inequality \eqref{congettura Xiao disuguaglianza} is false: Pirola in \cite{pirola}, resp.\ Albano and Pirola in \cite{albano pirola}, explicitly constructed 1 fibration, resp.\ 3 fibrations, that do not satisfy \eqref{congettura Xiao disuguaglianza}. Indeed, a modified version of the conjecture supposes $	q_h\le\lceil\frac{{\tilde g}+1}2\rceil.$

All the counterexamples are constructed considering families of covers in $\mathcal{R}(K, g, r)$. They have data:
	\begin{itemize}
	\item $ K=\ZZ/3\ZZ, g=1, r=3; $
	\item $ K=\ZZ/5\ZZ, g=2, r=0; $
	\item $ K=\ZZ/3\ZZ, g=4, r=0; $
	\item $ K=\ZZ/3\ZZ, g=3, r=0; $
	\end{itemize}
They all have constant Prym variety: indeed they have been found by considering families of coverings of hyperelliptic curves (elliptic curves in the example of Pirola \cite{pirola}) which lie in positive dimensional fibres of Prym maps. 

\begin{remark}\label{remark famiglia Pirola}
	The family studied by Pirola, i.e. the first one listed above, turns out to be interesting also from another point of view. Indeed, in \cite{fpp}, the authors show that the locus described by $J\tilde C$, for $\tilde C$ varying in the family, yields a Shimura subvariety of $\mathcal{A}_4$ generically contained in the Torelli locus. Indeed it satisfies the condition $(\ast)$ studied in the same paper which is sufficient to produce Shimura subvarieties generically contained in the Torelli locus. Moreover, in \cite{gm} and also in \cite{fgs}, it is proven that, via its Prym map, it is fibred in totally geodesic curves, countably many of which are Shimura. One of these Shimura fibres is the family (12) of \cite{fgp}. This is exactly the family we use to study the family of Pirola.  
\end{remark}

\begin{remark}
	The example with data $ K=\ZZ/5\ZZ, g=2, r=0 $ is curious in the same spirit of the previous Remark. Indeed, by \cite[Theorem 3.1]{ries}, it involves curves whose Jacobians have a non-trivial endomorphism algebra and the endomorphisms are not induced by the automorphisms of the curves. In \cite{spelta}, the second author shows that the Jacobians of such curves yield a new explicit Shimura subvariety of $\mathcal{A}_2$ generically contained in the Torelli locus. 
\end{remark}

In order to find new counterexamples to Xiao's conjecture, it seems quite natural to generalize the idea of \cite{pirola} and \cite{albano pirola} considering families of towers $\tilde C\ra C\ra \PP$ whose Prym map is constant but without requiring $C$ to be hyperelliptic and considering any Galois covering $\tilde C\ra C$ of any degree.
We have the following:
\begin{prop}\label{constroesempi}
	Up to $\tilde g=12, s=14$ (i.e. dimension 11), the only positive dimensional families of Galois towers $\tilde C\ra C\ra \PP$ that have $\tilde N - N<s-3$ and disprove \eqref{congettura Xiao disuguaglianza} are the family (12) of \cite{fgp} and the three examples of \cite{albano pirola}.
	\begin{proof} 
		Using Proposition \ref{proposizione diff non inj}, we know that when $\tilde N - N<s-3$ the differential of the Prym map associated with the family is not injective. Under this assumption, computer calculations in \verb|MAGMA| that impose $ q_h>\frac{{\tilde g}+1}2$ find only the four examples of the statement.\\
	\end{proof}
\end{prop}

Now we would like to explicitly describe the four examples as families of Galois towers $\tilde C\ra C\ra \PP$. 

The first example we treat is the family (12) of \cite{fgp}. Indeed the family of Pirola (\cite{pirola}) cannot be realized as a Galois cover of $\PP$. For this reason we will study the family (12) of \cite{fgp} which is contained in the family of Pirola and which describes Galois towers $\tilde C\ra C\ra \PP$, as already mentioned in Remark \ref{remark famiglia Pirola}. 
\begin{example}
$d=3$, $s=4$, ${\tilde g}=4$, $g=1$, $r=3$. \\
${\tilde G} = G(6,2) = \ZZ/6\ZZ = \sx{g:g^6=1}\xs$. \\
$(\tilde\theta(\gamma_1),\dots,\tilde\theta(\gamma_4))=(g^3 , g^5 , g^5 , g^5 ),\quad K=\sx{\sigma=g^2}\xs\cong\ZZ/3\ZZ$. The monodromy matrix is $A=(3,5,5,5)$. \\
The action of $\tilde G$ gives the decomposition $H^0(\tilde C,\om_{\tilde C})=W_3\oplus W_4\oplus W_5$, where $W_n$ is the subspace of $H^0(\tilde C,\omega_{\tilde C})$ where $\tilde G$ acts via the character $n$. We have $\dim W_3=\dim W_4=1$ and $\dim W_5=2$. 
Moreover $H^0(\tilde C,\om_{\tilde C})_+=W_3$. 
Therefore $(S^2H^0(\tilde C,\om_{\tilde C}))^{\tilde G}=(S^2H^0(\tilde C,\om_{\tilde C})_+)^{\tilde G}=S^2W_3$ and so the multiplication map $m: (S^2V_{-})^{\tilde G}\ra H^0(\om_{\tilde C}^{\otimes 2})^{\tilde G}$ is trivial. Hence the family is a curve that lies in a fibre of the Prym map $\mathcal{P}_{1,3}(3)$.
The relative irregularity of the fibred surface is $q_h={\tilde g}-g=3>\frac{5}{2}=\frac{{\tilde g}+1}{2}$, hence it violates the inequality \eqref{congettura Xiao disuguaglianza}.
\end{example}

Next example is the first one that appears in \cite[Section 4]{albano pirola}, given by an étale 5:1 cover. 
\begin{example}
$d=5$, $s=6$, ${\tilde g}=6$, $g=2$, $r=0$. \\
${\tilde G} = G(10,1) = D_5 = \sx{g_1,g_2:g_1^2=g_2^5=1,g_1g_2g_1=g_2^{-1}}\xs$. \\
$(\tilde\theta(\gamma_1),\dots,\tilde\theta(\gamma_6))=(g_1g_2^2 ,\ g_1g_2^4 ,\ g_1g_2 ,\ g_1g_2^4 ,\ g_1g_2 ,\ g_1g_2 ),\quad K=\sx{\sigma=g_2}\xs\cong\ZZ/5\ZZ$. \\
Using the notation of \texttt{MAGMA}, we get $H^0(\tilde C,\om_{\tilde C})=2V_2\oplus V_3\oplus V_4$, where $V_i$ are irreducible representations of $\tilde G$ such that $\dim V_2=1$ and $\dim V_3=\dim V_4=2$.
We have that $(S^2H^0(\tilde C,\om_{\tilde C}))^{\tilde G}=3S^2V_2\oplus(S^2V_3)^{\tilde G}\oplus(S^2V_4)^{\tilde G}$ has dimension 5 and that $H^0(\tilde C,\om_{\tilde C})_+=2V_2$. Therefore $\dim(S^2H^0(\tilde C,\om_{\tilde C})_{+})^{\tilde G}=3$.
Since $\dim(S^2H^0(\tilde C,\om_{\tilde C})_-)^{\tilde G}=2<s-3=3$, the Prym map is not injective on the family.
Again inequality \eqref{congettura Xiao disuguaglianza} does not hold: ${\tilde g}-g=4>\frac{7}{2}=\frac{{\tilde g}+1}{2}$.
\end{example}

Here we have the example of \cite[Section 5]{albano pirola}. 
\begin{example}
$d=3$, $s=10$, ${\tilde g}=10$, $g=4$, $r=0$. \\
${\tilde G} = G(6,1) = D_3 = \sx{g_1,g_2:g_1^2=g_2^3=1,g_1g_2g_1=g_2^{-1}}\xs$. \\
$(\tilde\theta(\gamma_1),\dots,\tilde\theta(\gamma_{10}))=( g_1, g_1, g_1g_2^2, g_1g_2^2, g_1g_2, g_1, g_1g_2^2, g_1g_2, g_1g_2, g_1 ),\quad K=\sx{\sigma=g_2}\xs\cong\ZZ/3\ZZ$. \\
In this case $H^0(\tilde C,\om_{\tilde C})=4V_2\oplus 3V_3$, where $\dim V_2=1$ and $\dim V_3=2$, $H^0(\tilde C,\om_{\tilde C})_+=4V_2$. Then $(S^2H^0(\tilde C,\om_{\tilde C}))^{\tilde G}=10S^2V_2\oplus 6(S^2V_3)^{\tilde G}$ has dimension 16 and $(S^2H^0(\tilde C,\om_{\tilde C})_+)^{\tilde G}=10S^2V_2$ has dimension 10. Hence $\tilde N-N=6<7$ guarantees $d\mathcal{P}_{4,0}(3)$ not injective. Moreover 
${\tilde g}-g=6>\frac{11}{2}=\frac{{\tilde g}+1}{2}$ violates \eqref{congettura Xiao disuguaglianza}.
\end{example}

Finally we have the example \cite[Section 6]{albano pirola}.  
\begin{example}
	$d=3$, $s=7$, ${\tilde g}=6$, $g=2$, $r=2$. \\
	${\tilde G} = G(6,1) = D_3 = \sx{g_1,g_2:g_1^2=g_2^3=1,g_1g_2g_1=g_2^{-1}}\xs$. \\
	$(\tilde\theta(\gamma_1),\dots,\tilde\theta(\gamma_7))=(g_1g_2 ,\ g_1g_2 ,\ g_1 ,\ g_1g_2^2 ,\ g_1 ,\ g_1 ,\ g_2 ),\quad K=\sx{\sigma=g_2}\xs\cong\ZZ/3\ZZ$. \\
	\texttt{MAGMA} gives $H^0(\tilde C,\om_{\tilde C})=2V_2\oplus 2V_3$, where $\dim V_2=1$ and $\dim V_3=2$ and $H^0(\tilde C,\om_{\tilde C})_+=2V_2$.
	We have that $(S^2H^0(\tilde C,\om_{\tilde C}))^{\tilde G}=3S^2V_2\oplus 3(S^2V_3)^{\tilde G}$ has dimension 6 and that $(S^2H^0(\tilde C,\om_{\tilde C})_+)^{\tilde G}=3S^2V_2$ has dimension 3. Once again the differential of the Prym map along the family is not injective and ${\tilde g}-g=4>\frac{7}{2}=\frac{{\tilde g}+1}{2}$ violates \eqref{congettura Xiao disuguaglianza}.
\end{example}

As already said, our data give a slightly different presentation of the last example presented by \cite{albano pirola}. We easily observe that they are the same. Indeed our family yields the following diagram:
\begin{equation}\label{diagramma famiglie}
\begin{tikzcd}
\tilde C\arrow{d}{2:1}[swap]{\pi}\arrow{r}{3:1}\arrow{dr}{\psi}&C\arrow{d}{2:1}\\
D\arrow{r}{3:1}&\PP\cong\tilde  C/{D_3},
\end{tikzcd}
\end{equation}
the curve $D$ is obtained as the quotient of $\tilde C$ by a lift of the hyperelliptic involution of $C$. We have 7 branch points $z_1,\dots,z_7$ in $\PP$ and the map $\tilde C\ra \PP$ has three ramification points of order 2 over $z_1,\dots,z_6$. Let us call them $p_{ij}, i=1,\dots,6, j=1,2,3$. Moreover $\psi$ has 2 ramification points of order three over $z_7$. Let us call them $q_1,q_2$. The 2:1 map $\tilde C\ra D$ ramifies on one among the  $p_{ij}$ for every $i$, while it is \'etale over the remaining two and it is \'etale over $q_1, q_2$. If we denote $p=\pi(q_1)=\pi(q_2)$, then we get the map $D\xrightarrow{3:1}\PP$ associated with the linear system $|3p|$. This is the starting point of the construction of Albano and Pirola. Indeed the map $\tilde C\ra C$ is \'etale over all the $p_{ij}'s$ while it is completely ramified over $q_1$ and $q_2$. When these two points are glued we get the singular curve $C_p$ of \cite{albano pirola}. The fibration of  \cite[Section 6]{albano pirola} is constructed desingularizing such curves  $C_p$, i.e.\ considering the curves $\tilde C$ provided by our example. Therefore this example is clearly another presentation of the one discussed in \cite[Section 6]{albano pirola}.

\end{document}